\documentclass[a4paper]{amsart}
\usepackage[latin1]{inputenc}
\usepackage{amssymb}
\usepackage{amsmath}
\usepackage{mathrsfs}
\usepackage{eufrak}
\usepackage{amsthm}
\usepackage{amsfonts}
\usepackage{textcomp}
\usepackage{graphicx}
\usepackage[pdftex]{color}
\usepackage{paralist}
\usepackage[shortlabels]{enumitem}
\usepackage{hyperref}
\usepackage{comment}
\usepackage[arrow, matrix, curve]{xy}

\newtheorem*{corollary*}{Corollary}
\newtheorem{theorem}{Theorem}[section]

\newtheorem{lemma}[theorem]{Lemma}
\newtheorem{proposition}[theorem]{Proposition}

\newtheorem*{claim*}{Claim}
\newtheorem*{conjecture}{Conjecture}

\theoremstyle{definition}
\newtheorem{definition}[theorem]{Definition}

\newtheorem{example}[theorem]{Example}

\theoremstyle{remark}

\numberwithin{equation}{theorem}

\makeatletter
\renewcommand*\env@matrix[1][\
arraystretch]{%
  \edef\arraystretch{#1}%
  \hskip -\arraycolsep
  \let\@ifnextchar\new@ifnextchar
  \array{*\c@MaxMatrixCols c}}
\makeatother

\begin{document}

\title{On a conjecture about dominant dimensions of algebras}
\date{\today}



\author{Ren\'{e} Marczinzik}
\address{Institute of algebra and number theory, University of Stuttgart, Pfaffenwaldring 57, 70569 Stuttgart, Germany}
\email{marczire@mathematik.uni-stuttgart.de}

\begin{abstract}
For every $n \geq 1$, we present examples of algebras $A$ having dominant dimension $n$, such that the algebra $B=End_A(I_0 \oplus \Omega^{-n}(A))$ has dominant dimension different from $n$, where $I_0$ is the injective hull of $A$. This gives a counterexample to conjecture 2 in \cite{CX}.
While the conjecture is false in general, we show that a large class of algebras containing higher Auslander algebras satisfies the property in the conjecture.
\end{abstract}

\maketitle
\section*{Introduction}
In this short article we provide a counterexample to the following conjecture stated in \cite{CX} as conjecture 2:
\begin{conjecture}
Let $A$ be a finite dimensional algebra with finite dominant dimension $n \geq 1$. 
Let $0 \rightarrow A \rightarrow I_0 \rightarrow I_1 \rightarrow ...$ be a minimal injective resolution of $A$, then $B:=End_A(I_0 \oplus \Omega^{-n}(A))$ has dominant dimension $n$.
\end{conjecture}
The motivation of this conjecture is due to the fact that Morita algebras (as defined in \cite{KerYam}) satisfy the conjecture, see \cite{CX} Corollary 4.19.
We show that this conjecture is false for every $n \geq 1$ by giving for any $n \geq 1$ an example of an algebra with dominant dimension $n$ such that $B$ as in the conjecture has dominant dimension strictly smaller than $n$. We say that algebras $A$ satisfying the conjecture have property *. While the conjecture is false in general, it might still be an interesting question to find another characterisation of algebras having property * and new classes of algebras having property *. We show that a large class of algebras containing higher Auslander algebras (as defined in \cite{Iya}) have property *. We also give a way to produce new higher Auslander algebras from old by using tilting modules related to dominant dimensions.
I thank Steffen Koenig for useful comments.
\section{Preliminaries}
In this article we always assume that $A$ is a finite dimensional, basic and connected algebra over a field $K$. To avoid trivialities, we assume that $A$ is not semisimple. Note that the assumption that our algebras are basic is not really restrictive since homological dimensions are invariant under Morita equivalence. We always work with finite dimensional right modules, if not stated otherwise. $mod-A$ denotes the category of finite dimensional right $A$-modules.
$D:=Hom_K(-,K)$ denotes the $K$-duality of an algebra $A$ over the field $K$.
For background on representation theory of finite dimensional algebras and their homological algebra, we refer to \cite{ASS}.
For a module $M$, $add(M)$ denotes the full subcategory of $mod-A$ consisting of direct summands of $M^n$ for some $n \geq 1$.
A module $M$ is called basic in case $M \cong M_1 \oplus M_2 \oplus ... \oplus M_n$, where every $M_i$ is indecomposable and $M_i$ is not isomorphic to $M_j$ for $i \neq j$. The basic version of a module $N$ is the unique (up to isomorphim) module $M$ such that $add(M)=add(N)$ and such that $M$ is basic.
We denote by $S_i=e_iA/e_iJ$, $P_i=e_i A$ and $I_i=D(Ae_i)$  the simple, indecomposable projective and indecomposable injective module, respectively, corresponding to the primitive idempotent $e_i$. \newline
The dominant dimension domdim($M$) of a module $M$ with a minimal injective resolution $(I_i): 0 \rightarrow M \rightarrow I_0 \rightarrow I_1 \rightarrow ...$ is defined as: \newline
domdim($M$):=$\sup \{ n | I_i $ is projective for $i=0,1,...,n \}$+1, if $I_0$ is projective, and \newline domdim($M$):=0, if $I_0$ is not projective. \newline
The codominant dimension of a module $M$ is defined as the dominant dimension of the $A^{op}$-module $D(M)$.
The dominant dimension of a finite dimensional algebra is defined as the dominant dimension of the regular module. It can be shown that the dominant dimension of an algebra always equals the dominant dimension of the opposite algebra, see for example \cite{Ta}.
So domdim($A$)$ \geq 1$ means that the injective hull of the regular module $A$ is projective or equivalently, that there exists an idempotent $e$ such that $eA$ is a minimal faithful projective-injective module.
Unless otherwise stated, $e$ without an index will always denote the idempotent such that $eA$ is the minimal faithful injective-projective $A$-module in case $A$ has dominant dimension at least one.
Algebras with dominant dimension larger than or equal to 1 are called QF-3 algebras.
All Nakayama algebras are QF-3 algebras (see \cite{Abr}, Proposition 4.2.2 and Propositon 4.3.3).
For more information on dominant dimensions and QF-3 algebras, we refer to \cite{Ta}.
\begin{definition}
$A$ is called a Morita algebra iff it has dominant dimension larger than or equal to 2 and $D(Ae) \cong eA$ as $A$-right modules. This is equivalent to $A$ being isomorphic to $End_B(M)$, where $B$ is a selfinjective algebra and $M$ a generator of mod-$B$ and in this case $B=eAe$ and $M=D(eA)$ (see \cite{KerYam}).
$A$ is called a gendo-symmetric algebra iff it has dominant dimension larger than or equal to 2 and $D(Ae) \cong eA$ as $(eAe,A)-$bimodules iff it has dominant dimension larger than or equal to 2 and $D(eA) \cong Ae$ as $(A,eAe)$-bimodules. This is equivalent to $A$ being isomorphic to $End_B(M)$, where $B$ is a symmetric algebra and $M$ a generator of mod-$B$ and in this case $B=eAe$ and $M=Ae$ (see \cite{FanKoe}).
\end{definition}
An algebra is called Gorenstein in case $injdim(A)=projdim(D(A))< \infty$. In this case $Gdim(A)$ is called the Gorenstein dimension of $A$ and we say that $A$ has infinite Gorenstein dimension if $injdim(A)= \infty$. Note that $Gdim(A)= max \{ injdim(e_iA) | e_i$ a primitive idempotent$ \}$ and $domdim(A)= min \{ domdim(e_iA) | e_i $ a primitive idempotent $ \}$.
An algebra $A$ is called higher Auslander algebra in case $\infty>domdim(A)=gldim(A) \geq 2$, see \cite{Iya}.
We now recall some results on Nakayama algebras. See chapter 32 in \cite{AnFul} or chapter 5 in \cite{ASS} for more on this topic.
Let $A$ always be a finite dimensional connected Nakayama algebra given by quiver and relations for the rest of this section. Thus their quiver is a directed line or a directed circle. We choose to number the points in the quiver by $0,1,...,n-1$ in a clockwise way in case the algebra has $n$ simple modules.
In this case, the algebra is uniquely determined by the sequence $c=(c_0,c_1,...,c_{n-1})$ (see \cite{AnFul}, Theorem 32.9.), where $c_i$ denotes the dimension of the indecomposable projective module $e_iA$ and $n$ is the number of simple modules. The sequence $(c_0,c_1,...,c_{n-1})$ is called the Kupisch series of $A$. We look at the indices of the $c_i$ always modulo $n$. Thus $c_i$ is defined for every $i \in \mathbb{Z}$. For a given Nakayama algebra with $n$ simple modules, $d=(d_0,d_1,...,d_{n-1})$ denotes the CoKupisch series, where $d_i:=dim(D(Ae_i))$ is the dimension of a indecomposable injective $A$-module.
Every indecomposable module over a Nakayama algebra is isomorphic to $e_iA/e_iJ^k$ for some $k \in \{1,2,...,c_i\}$. For $k=c_i$, one gets exactly the indecomposable projective modules.

\begin{lemma}
\label{Lengthinj}
The dimension of the indecomposable projective left module $Ae_i$ at a vertex $i$ (and, therefore, the length of the indecomposable injective right module at $i$) satisfies:
$$d_i=\inf \{k \geq 1 | k \geq c_{i-k} \}.$$
\noindent Furthermore, the values $c_i$ are a permutation of the values of the $d_j$.
\end{lemma}
\begin{proof}
See \cite{Ful} Theorem 2.2. 
\end{proof}

\begin{lemma}
Let $M:= e_iA/e_iJ^m$ be an indecomposable module of the Nakayama algebra $A$ with $m=dim(M) \leq c_i$.
Then $M$ is injective iff $c_{i-1} \leq m$. Especially: $e_iA$ is injective iff $c_{i-1} \leq c_i$.
\end{lemma}
\begin{proof}
See \cite{AnFul} Theorem 32.6. 
\end{proof}

Now we give a fast method to calculate minimal injective resolutions in Nakayama algebras. Let $M:=e_iA/e_iJ^k$ be an indecomposable $A$-module.
We get a minimal injective resolution of $M$ as follows:
We have soc($M$)=$e_iJ^{k-1}/e_iJ^k \cong S_{i+k-1}$ (the simple module corresponding to the point $i+k-1$). Therefore, the injective hull of $M$ is $D(Ae_{i+k-1})$ and $\Omega^{-1}(M)=D(J^k e_{i+k-1})$ by looking at dimensions and using that submodules form a chain. 
Now the socle of $D(J^k e_{i+k-1})$ equals the top of $J^k e_{i+k-1}$, which is $S_{i-1}$. Therefore, the injective hull of $D(J^k e_{i+k-1})$ is $D(Ae_{i-1})$ and $\Omega^{-1}(D(J^k e_{i+k-1}))=D(J^{d_{i+k-1}-k} e_{i-1})$ again by looking at dimensions and using that submodules form a chain. 
If we denote $D(J^y e_x)$ for short by $[x,y] \in \mathbb{Z}/n \times \mathbb{N}$ then we get that $\Omega^{-1}(D(J^y e_x))=\Omega^{-1}([x,y])=[x-y,d_x-y]$. Like this we can calculate the cosyzygies and the minimal inejctive resolution of an indecomposable module over a Nakayama algebra successively. Note that the cosyzygies determine the minimal injective resolution completely.
We will need the following two lemmas later:
\begin{lemma}
\label{umrechnung}
$D(J^k e_i) \cong e_{i-d_i+1}A/e_{i-d_i+1}J^{d_i-k}$.
\end{lemma}
\begin{proof}
Since $J^k e_i$ is indecomposable, $D(J^k e_i)$ is also indecomposable. We first calculate the top of the module:
$top(D(J^ke_i)) \cong D(soc(J^ke_i)) \cong D(J^{d_i-1}e_i)=S_{i-d_i+1}$. Thus there exists a short exact sequence: \newline
$0 \rightarrow e_{i-d_i+1}J^s \rightarrow e_{i-d_i+1}A \rightarrow D(J^ke_i) \rightarrow 0$, for some $s$. Comparing dimensions, one obtains $s=d_i-k$ and the isomorphism follows.
\end{proof}
\begin{lemma}
\label{homspace}
$Hom_A(e_iA/e_iJ^k,e_jA/e_jJ^l) \cong (e_jJ^{max(0,l-k)}/e_jJ^l)e_i$, where we interpret $J^0=A$.
\end{lemma}
\begin{proof}
Recall that an isomorphism $(e_jA/e_jJ^l)e_i \cong Hom_A(e_iA,e_jA/e_jJ^l) $ is given by $z \rightarrow l_z$, where $l_z$ denotes left multiplication by the element $z$. Now to be in \newline\noindent $Hom_A(e_iA/e_iJ^k,e_jA/e_jJ^l)$ corresponds to the condition $l_z(e_iJ^k)=0$, which translates into $ze_iJ^k=0$. Since $z$ is in $(e_jA/e_jJ^l)e_i$, $ze_iJ^k=0$ is equivalent to $z \in (e_jJ^{max(0,l-k)}/e_jJ^l)e_i$. 
\end{proof}
Note that in case $l=c_j$, the formula simplifies to $e_jJ^{max(0,l-k)}e_i$.

In the rest of this article we will mostly deal with Nakayama algebras with $n$ simple modules and $c_{n-1}=1$. Thus their quiver looks as follows:
$$Q=\begin{xymatrix}{ \circ^{0} \ar[r]^{\alpha_0} & \circ^{1} \ar[r]^{\alpha_1}  & \circ^{2} \ar @{} [r] |{\cdots} & \circ^{n-2} \ar[r]^{\alpha_{n-2}}  & \circ^{n-1}}\end{xymatrix}.$$
\section{Counterexample to a conjecture concerning dominant dimensions}
\begin{definition}
Let $A$ be a finite dimensional algebra with finite dominant dimension $n \geq 1$. 
We say $A$ has \textit{property *}, iff the following is satisfied:
Let $0 \rightarrow A \rightarrow I_0 \rightarrow I_1 \rightarrow ...$ be a minimal injective resolution of $A$, then $B:=End_A(I_0 \oplus \Omega^{-n}(A))$ has dominant dimension $n$.
\end{definition}
In \cite{CX}, the authors have shown that all Morita algebras have property *.
In \cite{CX} section 5.2., one can find the following as conjecture 2:
\begin{conjecture}
Every finite dimensional algebra with finite dominant dimension $n \geq 1$ has property *.
\end{conjecture}
One interest in this conjecture stems from the fact that in case $A$ has dominant dimension $n \geq 1$, the modules $I_0 \oplus \Omega^{-i}(A)$ are tilting modules for $1 \leq i \leq n$. That those modules are really tilting modules is mentioned in \cite{CX} and we give the easy proof in the next section.
We found a counterexample to this conjecture in case $A$ has dominant dimension 1. Changchang Xi then posed the question wheter there exists counterexamples of arbitrary dominant dimension $n \geq 1$. We will find a class of algebras answering his question.
\begin{lemma}
Let $A$ be an algebra with dominant dimension $d \geq 1$ and minimal faithful projective-injective module $eA$.
Let $I_0$ be the injective hull of $A$. Then $End_A(I_0 \oplus \Omega^{-d}(A))$ and $End_A(eA \oplus \Omega^{-d}(A))$ are Morita-equivalent and thus have the same dominant dimensions and Gorenstein dimensions.
\end{lemma}
\begin{proof}
Just note that $add(I_0)=add(eA)$, since $A$ has dominant dimension at least one. Thus $add(I_0 \oplus \Omega^{-d}(A))=add(eA \oplus \Omega^{-d}(A))$ and therefore both algebras are Morita-equivalent by Lemma 6.12. of \cite{SkoYam}. The dominant and Gorenstein dimension are preserved, since a Morita equivalence $F$ between two algebras $A$ and $B$ sends a minimal injective resolution of the basic version of the regular module of $A$ to a minimal injective resolution of the basic version of the regular module of $B$ and furthermore $F$ preserves injective-projective modules.
\end{proof}
Because of the previous lemma, we will from now on always look at the algebra $End_A(eA \oplus \Omega^{-d}(A))$ instead of $End_A(I_0 \oplus \Omega^{-d}(A))$.
\begin{definition}
Let $n \geq 2$.
Let $A_n$ be defined as the connected quiver Nakayama algebra with Kupisch series $[c_0,c_1,...,c_{n-1}]$, where $c_i=3$ for $i=0,...,n-3$ and $c_{n-3+i}=3-i$ for $i=1$ and $i=2$. Let $\hat{d_n}$ denote the dominant dimension of $A_n$.
Then we define the algebra $B_n:=End_{A_n}(e_nA_n \oplus \Omega^{-\hat{d_n}}(A_n))$, where $e_nA_n$ is the minimal faithful projective-injective $A_n$-module.
\end{definition}
We keep the notation for $A_n$ and $B_n$ for this rest of this section.

We note that $A_n$ has for the dimension of the injective indecomposable modules just $c$ reverted: $d=[1,2,3,3,...,3]$, as can be easily checked using \hyperref[Lengthinj]{ \ref*{Lengthinj}}.

\begin{proposition}
\label{domdimA}
For $m \geq 1$, the following holds: \newline
1.$domdim(A_{3m})=gldim(A_{3m})=2m-1$. \newline
2.$domdim(A_{3m+1})=2m-1$ and $gldim(A_{3m+1})=2m$ and \newline $\Omega^{-(2m-1)}(A_{3m+1})=e_1 A/e_1J^2 \oplus e_1A/e_1 J^1$. \newline
3.$domdim(A_{3m+2})=2m$ and $gldim(A_{3m+2})=2m+1$ and \newline $\Omega^{-2m}(A_{3m+2})=e_0 A/e_0 J^2 \oplus e_1A/e_1 J^1$.
\end{proposition}
\begin{proof}
Note first that $A=A_n$ has exactly $n-2$ projective-injective indecomposable modules, namely $e_0A ,e_1A,...,e_{n-3}A$ because $3=c_{i-1} \leq c_i=3$ holds in those cases. Furthermore, $e_iA \cong D(Ae_{i+2})$ in this case. Thus we only have to calculate the dominant dimensions and injective dimensions of the projective modules $e_{n-2}A$ and $e_{n-1}A$. We calculate the relevant cosyzygies and then the dominant and injective dimensions can be read off: 
We have $\Omega^{-1}(e_{n-2}A)=[n-1,2]$ and using our formula from the preliminaries to calculate cosyzygies and using induction one has $\Omega^{-(2l+1)}(e_{n-2}A)=[n-(3l+1),2]$ for $0 \leq l$, such that $n-(3l+1) \geq 2$ and $\Omega^{-(2l)}(e_{n-2}A)=[n-3l,1]$ for $1 \leq l$, such that $n-3l \geq 2$.
Similarly, $\Omega^{-1}(e_{n-1}A)=[n-1,1]$ and $\Omega^{-(2l+1)}=[n-(3l+1),1]$ for $0 \leq l$, such that $n-(3l+1) \geq 2$ and $\Omega^{-(2l)}(e_{n-2}A)=[n-(3l-1),1]$ for $0 \leq l$, such that $n-(3l-1) \geq 2$. \newline
Now we begin the proof 1.,2. and then 3.: \newline
1. Setting $n=3m$, we get that $\Omega^{-s}(e_{3m-2}A)=[i_s,k_s]$ with $i_s \notin \{2,3,...,n-1 \}$ for the first time for $s=2m$ and thus $domdim(e_{3m-2}A)=2m-1$ in this case. We get that $\Omega^{-s}(e_{3m-2}A)=[i_s,k_s]$ with $k_s=0$ for the first time for $s=2m+1$ and thus $injdim(e_{3m-2}A)=2m-1$ in this case. Similarly, we get that $\Omega^{-s}(e_{3m-1}A)=[i_s,k_s]$ with $i_s \notin \{2,3,...,n-1 \}$ for the first time for $s=2m$ and thus $domdim(e_{3m-1} A)=2m-1$ in this case. We get that $\Omega^{-s}(e_{3m-1}A)=[i_s,k_s]$ with $k_s=0$ for the first time for $s=2m+1$ and thus $injdim(e_{3m-1}A)=2m-1$ in this case. \newline
2. The cosyzygies at the end of the minimal injective resolution of $e_{3m-1}A$ look as follows: $[3,2] \rightarrow [1,1] \rightarrow [0,1] \rightarrow [-1,0]$, where $[3,2]=\Omega^{-(2m-1)}(e_{3m-1}A)$. Thus $domdim(e_{3m-1}A)=2m-1$ and $injdim(e_{3m-1}A)=2m$.
The cosyzygies at the end of the minimal injective resolution of $e_{3m}A$ look as follows: $[3,1] \rightarrow [2,2] \rightarrow [0,1] \rightarrow [-1,0]$, where $[3,1]=\Omega^{-(2m-1)}(e_{3m}A)$. Thus $domdim(e_{3m}A)=2m$ and $injdim(e_{3m}A)=2m$. Thus $domdim(A_{3m+1})=2m-1$ and $injdim(A_{3m+1})=2m$. $\Omega^{-(2m-1)}(A_{3m+1})=[3,1] \oplus [3,2]$. Note that $[3,1]=D(J^1e_3) \cong e_1 A/e_1J^2$ and $[3,2]=D(J^2 e_3) \cong e_1A/e_1 J^1 \cong S_1$, by \hyperref[umrechnung]{ \ref*{umrechnung}}. \newline
3. The cosyzygies at the end of the minimal injective resolution of $e_{3m}A$ look as follows: $[4,2] \rightarrow [2,1] \rightarrow [1,2] \rightarrow [-1,0]$, where $[4,2]=\Omega^{-(2m-1)}(e_{3m}A)$. Thus $domdim(e_{3m}A)=2m$ and $injdim(e_{3m}A)=2m$.
The end cosyzygies at the of the minimal injective resolution of $e_{3m+1}A$ look as follows: $[4,1] \rightarrow [3,2] \rightarrow [1,1] \rightarrow [0,1] \rightarrow [-1,0]$, where $[4,1]=\Omega^{-(2m-1)}(e_{3m+1}A)$. Then $domdim(e_{3m-1}A)=2m$ and $injdim(e_{3m-1}A)=2m+1$. Thus $domdim(A_{3m+2})=2m$ and $injdim(A_{3m+2})=2m+1$. $\Omega^{-(2m)}(A_{3m+1})=[2,1] \oplus [3,2]$. Note that $[2,1]=D(J^1e_3) \cong e_0 A/e_0 J^2 $ and $[3,2]=D(J^2 e_3) \cong e_1A/e_1 J^1 \cong S_1$, using \hyperref[umrechnung]{ \ref*{umrechnung}}. 
\end{proof}
We note that the algebras $A_n$ for $n \neq 0$ mod $n$ can have any positive integer as dominant dimension and thus those algebras are candidates for examples of algebras with arbitrary dominant dimensions not having property *. In the rest of this section we will show that those algebras indeed do not have property *.
We now concentrate on the algebras $A_n$ with $n \neq 0$ mod $n$, since we will look at the case $n \equiv 0$ mod $n$ in more generality in the next section.

Now we calculate the quiver and relations of the the algebras $B_n$ for $n \neq 0$ mod $n$ in order to calculate their dominant dimensions:
\begin{lemma}
Let $m \geq 1$. \newline
1. $B_{3m+1}$ is isomorphic to the quiver algebra $C_1$, which looks as follows ($n=3m+1$): \newline
$\xymatrix@1{  & \circ^{n-1}\ar [d]^{\gamma} &  &  & & &  &  \\ \circ^{0}\ar[r]^{\beta_0} & \circ^{1}\ar[r]^{\beta_1} & \circ^2\ar[r]^{\beta_2} & \circ^3\ar[r]^{\beta_3} & \circ^4\ar@{} [r] |(.45){\cdots} & \circ^{n-3}\ar[r]^{\beta_{n-3}} & \circ^{n-2}}$ \newline and having the relations that all paths of length larger or equal to three vanish expect the path $\gamma \beta_1 \beta_2$. \newline
2. $B_{3m+2}$ is isomorphic to the quiver algebra $C_2$, which looks as follows ($n=3m+2$): \newline
\noindent $\xymatrix@1{ \circ^{n-1}\ar [r]^{\alpha} \ar [d]^{\delta} & \circ^{n-2}\ar [d]^{\gamma} &  &  & & &  &  \\ \circ^{0}\ar[r]^{\beta_0} & \circ^{1}\ar[r]^{\beta_1} & \circ^2\ar[r]^{\beta_2} & \circ^3\ar[r]^{\beta_3} & \circ^4\ar@{} [r] |(.45){\cdots} & \circ^{n-4}\ar[r]^{\beta_{n-4}} & \circ^{n-3}}$
 \newline and having the relations that all paths of length larger or equal to three vanish and additionally $\beta_0 \beta_1=0$ and $\alpha \gamma=\delta \beta_0$.
\end{lemma}
\begin{proof}
We write $\alpha_i$ for the arrows in $A_n$ as in the quiver from the end of the preliminaries. \newline
1.First let $n=3m+1$ and $A=A_{3m+1}$ and note that $A$ has dimension $9m$. We first view $B_{3m+1}$ as a matrix algebra: \newline
$B_{3m+1}= End_A(eA \oplus e_1 A/e_1J^2 \oplus e_1A/e_1 J^1)= \newline
\begin{pmatrix}[1]
  Hom(eA,eA) & Hom(e_1A/e_1J^2,eA) & Hom(e_1A/e_1J^1,eA) \\
  (e_1 A/e_1J^2)e & Hom(e_1A/e_1J^2,e_1A/e_1J^2) & Hom(e_1A/e_1J^1,e_1A/e_1J^2) \\
  (e_1A/e_1J^1)e & Hom(e_1 A/e_1 J^2,e_1A/e_1 J^1) & Hom(e_1A/e_1J^1,e_1A/e_1J^1)
 \end{pmatrix}
 $.
 Calculating the individual entries using \hyperref[homspace]{ \ref*{homspace}}, we get the following: \newline
 $B_{3m+1}= 
\begin{pmatrix}[1]
  eAe & eJe_1 & 0 \\
  (e_1 A/e_1J^2)e & e_1 A e_1/e_1 J^2e_1 & 0 \\
  (e_1A/e_1J^1)e & (e_1A/e_1J^1)e_1 & e_1Ae_1/e_1J^1e_1
 \end{pmatrix}
 $.
$eAe$ is just the Nakayama algebra with Kupisch series $[3,3,3,...,3,2,1]$, with two $3$'s less than in the Kupisch series of $A_{3m+1}$. $eAe$ then has dimension $9m-6$ and $B_{3m+1}$ has dimension $9m+1$.
It follows that $rad(B_{3m+1})= \newline
\begin{pmatrix}[1]
  eJe & eJe_1 & 0 \\
  (e_1 A/e_1J^2)e & 0 & 0 \\
  (e_1A/e_1J^1)e & (e_1A/e_1J^1)e_1 & 0
 \end{pmatrix}
 $.
Then we get $rad^2(B_{3m+1})= \newline
\begin{pmatrix}[1]
  eJe & eJe_1 & 0 \\
  (e_1 A/e_1J^2)e & 0 & 0 \\
  (e_1A/e_1J^1)e & (e_1A/e_1J^1)e_1 & 0
 \end{pmatrix}
 \cdot 
 \begin{pmatrix}[1]
  eJe & eJe_1 & 0 \\
  (e_1 A/e_1J^2)e & 0 & 0 \\
  (e_1A/e_1J^1)e & (e_1A/e_1J^1)e_1 & 0
 \end{pmatrix}
 = \newline
 \begin{pmatrix}[1]
  eJeJe+eJe_1(e_1A/e_1J^2)e & 0 & 0 \\
  (e_1 A/e_1J^2)eJe & 0 & 0 \\
  (e_1A/e_1J^1)e & 0 & 0
 \end{pmatrix}
$.
Now we can calculate $rad(B_{3m+1})/rad^2(B_{3m+1})$ having a respective $K$-basis as entries: \newline
$rad(B_{3m+1})/rad^2(B_{3m+1})= 
 \begin{pmatrix}[1]
 <\alpha_1, \alpha_2,...,\alpha_{n-4}> & <\hat{\alpha_0}> & 0 \\
  <\hat{e_1}> & 0 & 0 \\
  0 & <\hat{e_1}> & 0
 \end{pmatrix} $. \newline
Here an hat over an element denotes that it is a rest class in some factor module.
Now for $i=2,3,...,n-3$ set $\beta_i=  
\begin{pmatrix}[1]
 \alpha_{i-1} & 0 & 0 \\
  0 & 0 & 0 \\
  0 & 0 & 0
 \end{pmatrix}$ and $\beta_0=
  \begin{pmatrix}[1]
 0 & \hat{\alpha_0} & 0 \\
  0 & 0 & 0 \\
  0 & 0 & 0
 \end{pmatrix}$, $\beta_1=
  \begin{pmatrix}[1]
 0& 0 & 0 \\
  \hat{e_1} & 0 & 0 \\
  0 & 0 & 0
 \end{pmatrix}$ and $\gamma=
  \begin{pmatrix}[1]
 0 & 0 & 0 \\
  0 & 0 & 0 \\
  0 & \hat{e_1} & 0
 \end{pmatrix}$.
Let $C_1$ be the quiver algebra of the above quiver with relations. Then $C_1$ has dimension $9m+1$ and surjects into $B_{3m+1}$ by sending arrows in $C_1$ to the same named arrows in $B_{3m+1}$, since those relations are also satisfied in the matrix algebra $B_{3m+1}$ using the corresponding arrows in the above presentation of $rad(B_{3m+1})/rad^2(B_{3m+1})$. Comparing dimension of $C_1$ and $B_{3m+1}$, one sees that $C_1$ is isomorphic to $B_{3m+1}$. \newline
2. First let $n=3m+2$ and $A=A_{3m+2}$ and note that $A$ has dimension $9m+3$. We first view $B_{3m+2}$ as a matrix algebra: \newline
$B_{3m+2}= End_A(eA \oplus e_0 A/e_0 J^2 \oplus e_1A/e_1 J^1)= \newline
\begin{pmatrix}[1]
  Hom(eA,eA) & Hom(e_0A/e_0J^2,eA) & Hom(e_1A/e_1J^1,eA) \\
  (e_0 A/e_0J^2)e & Hom(e_0A/e_0J^2,e_0A/e_0J^2) & Hom(e_1A/e_1J^1,e_0A/e_0J^2) \\
  (e_1A/e_1J^1)e & Hom(e_0 A/e_0 J^2,e_1A/e_1 J^1) & Hom(e_1A/e_1J^1,e_1A/e_1J^1)
 \end{pmatrix}
 $.
 Calculating the individual entries using \hyperref[homspace]{ \ref*{homspace}}, we get the following: \newline
 $B_{3m+2}= 
\begin{pmatrix}[1]
  eAe & 0 & 0 \\
  (e_0 A/e_0J^2)e & e_0 A e_0/e_0 J^2e_0 & (e_0A/e_0J^2)e_1 \\
  (e_1A/e_1J^1)e & 0 & e_1Ae_1/e_1J^1e_1
 \end{pmatrix}
 $.
$eAe$ is just the Nakayama algebra with Kupisch series $[3,3,3,...,3,2,1]$, with two $3$'s less than in the Kupisch series of $A_{3m+2}$. $eAe$ then has dimension $9m-3$ and $B_{3m+1}$ has dimension $9m-3+6=9m+3$.
It follows that $rad(B_{3m+1})= \newline
\begin{pmatrix}[1]
  eJe & 0 & 0 \\
  (e_0 A/e_0J^2)e & 0 & (e_0A/e_0J^2)e_1 \\
  (e_1A/e_1J^1)e & 0 & 0
 \end{pmatrix}
 $.
Then we get $rad^2(B_{3m+1})= \newline
\begin{pmatrix}[1]
  eJe & 0 & 0 \\
  (e_0 A/e_0J^2)e & 0 & (e_0A/e_0J^2)e_1 \\
  (e_1A/e_1J^1)e & 0 & 0
 \end{pmatrix}
 \cdot 
\begin{pmatrix}[1]
  eJe & 0 & 0 \\
  (e_0 A/e_0J^2)e & 0 & (e_0A/e_0J^2)e_1 \\
  (e_1A/e_1J^1)e & 0 & 0
 \end{pmatrix}
 = \newline
 \begin{pmatrix}[1]
  eJeJe & 0 & 0 \\
  (e_0 A/e_0J^2)eJe +(e_0A/e_0J^2)e_1(e_1A/e_1J^1)e & 0 & 0 \\
  (e_1A/e_1J^1)eJe & 0 & 0
 \end{pmatrix}
$. \newline
Now we can calculate $rad(B_{3m+2})/rad^2(B_{3m+2})$ having a respective $K$-basis as entries: \newline
$rad(B_{3m+2})/rad^2(B_{3m+2})= 
 \begin{pmatrix}[1]
 <\alpha_0,\alpha_1, \alpha_2,...,\alpha_{n-4}> & 0 & 0 \\
  <\hat{\alpha_0}> & 0 & <\hat{\alpha_0}> \\
  <\hat{e_1}> & 0 & 0
 \end{pmatrix} $. \newline
Here an hat over an element denotes that it is a rest class in some factor module.
Now for $i=0,1,...,n-4$ set $\beta_i=  
\begin{pmatrix}[1]
 \alpha_i & 0 & 0 \\
  0 & 0 & 0 \\
  0 & 0 & 0
 \end{pmatrix}$ and $\gamma=
  \begin{pmatrix}[1]
 0 & 0 & 0 \\
  \hat{\alpha_0} & 0 & 0 \\
  0 & 0 & 0
 \end{pmatrix}$, $\alpha=
  \begin{pmatrix}[1]
 0& 0 & 0 \\
  0 & 0 & \hat{\alpha_0} \\
  0 & 0 & 0
 \end{pmatrix}$ and $\delta=
  \begin{pmatrix}[1]
 0 & 0 & 0 \\
  0 & 0 & 0 \\
  \hat{e_1} & 0 & 0
 \end{pmatrix}$. \newline
Let $C_2$ be the quiver algebra of the above quiver with relations. Then $C$ has dimension $9m+3$ and surjects into $B_{3m+1}$ by sending arrows in $C_2$ to the same named arrows in $B_{3m+1}$, since those relations are also satisfied in the matrix algebra $B_{3m+2}$ using the corresponding arrows in the above presentation of $rad(B_{3m+2})/rad^2(B_{3m+2})$. Comparing dimension of $C_2$ and $B_{3m+2}$, one sees that $C_2$ is isomorphic to $B_{3m+2}$.

\end{proof}

\begin{proposition}
Let $m \geq 1$. \newline
1. $domdim(B_{3m+1})=0$. \newline
2. $domdim(B_{3m+2})=1$. \newline
Thus the algebras $A_{3m+1}$ and $A_{3m+2}$ do not have property *.
\end{proposition}
\begin{proof}
1. We use the presentation as quiver and relations of $B=B_{3m+1}$, given in the previous lemma.
We show that $domdim(e_0B)=0$. For this note that $soc(e_0B)=e_0J^2=S_2$ and thus the injective hull of $e_0B$ is $D(Be_2)$.
To show that $D(Be_2)$ is not projective, use $top(D(Be_2))=Dsoc(Be_2)=D(S_0 \oplus S_{n-1})=S_0 \oplus S_{n-1}$. Thus the indecomposable module $D(Be_2)$ does not have a simple top and thus can not be projective. Thus $domdim(B)=0$ in this case. \newline
2. We use the presentation as quiver and relations of $B=B_{3m+2}$ derived in the previous lemma.
First we show that $e_iB$ is injective iff $i \neq 0,n-4,n-3$. 
Note that $e_{n-1}B$ has dimension 4 and $soc(e_{n-1}B)=e_{n-1}J^2 \cong S_1$.
Thus $e_{n-1}B$ injects into $D(Be_1)$ but since $dim(D(Be_1))=4$, one has $e_{n-1}B \cong D(Be_1)$.
Similarly $e_{n-2}B$ has dimension 3 and $soc(e_{n-2}B)=e_{n-2}J^2 \cong S_2$.
Thus $e_{n-2}B$ injects into $D(Be_2)$ but since $dim(D(Be_2))=3$, one has $e_{n-2}B \cong D(Be_2)$.
Now let $i \in \{1,2,3,...,i-5 \}$. Then $e_iB$ has dimension 3 and $soc(e_iB)=e_iJ^2 \cong S_{i+2}$.
Thus $e_i B$ injects into $D(Be_{i+2})$ but since $dim(D(Be_{i+2}))=3$, one has $e_iB \cong D(Be_{i+2})$.
$e_{n-4}B$ has dimension two and $soc(e_{n-4}B) = e_{n-4}J \cong S_{n-3}$. Therefore $e_{n-4}B$ injects into $D(Be_{n-3})$ (this injection is no isomorphism since $dim(D(Be_{n-3}))=3>2$), which we saw is projective. Thus $domdim(e_{n-4}B) \geq 1$.
Clearly, $e_{n-3}B \cong S_{n-3}$ injects into $D(Be_{n-3})$ and thus also $domdim(e_{n-3}B) \geq 1$.
Now $e_0B$ has dimension two and $soc(e_0B)=e_0J \cong S_1$. Thus $e_0B$ injects into $D(Be_1) \cong e_{n-1}B$ and comparing dimensions, one has $e_0B \cong e_{n-1}J$. Now, the injective hull of $e_{n-1}B/e_{n-1}J \cong S_{n-1}$ is $D(Be_{n-1})$, which is not projective.
Thus $domdim(e_0B)=1$ and therefore $domdim(B)=1$.  
\end{proof}

We note that $A_n$ and $B_n$ are derived equivalent, since $eA \oplus \Omega^{-\hat{d_n}}(A)$ is a tilting $A$-module and thus the dominant dimension of derived equivalent algebras can differ by an arbitrary large number.
\section{Property * for higher Auslander-Solberg algebras}
We recall the definition of a tilting module. We refer for example to \cite{Rei} for statements with no proof here.
\begin{definition}
Let $m \geq 1$ be a natural number.
Let $T$ be an $A$-module, then $T$ is called an $m$-tilting module in case it has the following three properties: \newline
1. $projdim(T) \leq m$ \newline
2. $Ext^{i}(T,T)=0$ for every $i \geq 1$ \newline
3. There exists an exact sequence of the form \newline $0 \rightarrow A \rightarrow T_0 \rightarrow T_1 \rightarrow ... \rightarrow T_m \rightarrow 0$ with $T_i \in add(T)$.
\end{definition}
We often say the shorter term tilting module instead of $m$-tilting module when $m$ does not matter that much.
We note that assuming conditions 1. and 2.,  the third condition is equivalent to the basic version of $T$ having exactly $s$ indecomposable summands, where $s$ denotes the number of simple $A$-modules. 
In the previous section we did not answer the question wheter the algebras $A_n$ for $n \equiv 0$ mod 3 have property *. Note that by \hyperref[domdimA]{ \ref*{domdimA}} 1. $A_n$ is a higher Auslander algebra in this case.
In this section we prove that algebras with Gorenstein dimension equal to the dominant dimension $domdim(A) \geq 2$, which we call \textit{higher Auslander-Solberg algebras} (the motivation for this name is the paper \cite{AS}, where such algebras seem to appear for the first time) have property *. This class of algebras clearly generalize higher Auslander algebras, that is algebras $A$ having global dimension equal to dominant dimension $dimdim(A) \geq 2$. To see this, just note that the Gorenstein dimension equals the global dimension, in case the global dimension is finite. Chen and Koenig recently found a characterization when an algebra of the form $End_C(M)$ is a higher Auslander-Solberg algbra, where $C$ is some algebra and $M$ a generator-cogenerator of $mod-C$, see \cite{CheKoe}.
Part 1 of the following lemma is mentioned in \cite{CX} without proof:
\begin{lemma}
Let $n$ be a positive natural number and $i \in \{1,2,...,n\}$.
Let $A$ be an algebra with dominant dimension $n \geq 1$ and minimal faithful projective-injective module $eA$. \newline
1. A module of the form $T=eA \oplus \Omega^{-i}(A)$ is a tilting module. \newline
2. For an indecomposable projective but noninjective module $e_jA$ the basic version of the module $\Omega^{-i}(e_jA)$ is indecomposable.
\end{lemma}
\begin{proof}
1. Conditions 1. and 3. in the definition of a tilting module are clear since the start of a minimal injective resolution of $A$ looks as follows: \newline
$0 \rightarrow A \rightarrow I_0 \rightarrow I_1 \rightarrow ... \rightarrow I_{i-1} \rightarrow \Omega^{-i}(A) \rightarrow 0$ and $I_j \in add(eA)$ since $A$ is assumed to have dominant dimension at least $n$. What is left to show is that $Ext^{k}(T,T)=0$ for every $k \geq 1$.
Now note that since $A$ has dominant dimension $n$, $\Omega^{-i}(A)=\Omega^{n-i}(\Omega^{-n}(A))$ and thus 
$Ext^{k}(eA \oplus \Omega^{-i}(A), eA \oplus \Omega^{-i}(A))$ (here we use that $eA$ is projective and injective)$ \newline =Ext^{k}(\Omega^{-i}(A),\Omega^{-i}(A)) = Ext^{i+k}(\Omega^{-i}(A),A)=Ext^{i+k}(\Omega^{n-i}(\Omega^{-n}(A)),A)= \newline Ext^{n+k}(\Omega^{-n}(A),A)=0$, since $\Omega^{-n}(A)$ has projective dimension $n$. \newline
2. This follows from 1. and the fact that $eA \oplus \Omega^{-i}(A)$ is a tilting module and thus its basic version has exactly $s$ simple modules, where $s$ is the number of simple $A$-modules. 
\end{proof}
Following \cite{CX}, we call tilting modules of the form $eA \oplus \Omega^{-i}(A)$ canonical tilting modules for $i \in \{1,2,...,domdim(A) \}$.
\begin{theorem}
Let $A$ be an algebra with finite dominant dimension $domdim(A) \geq 2$ and $Gordim(A)=domdim(A)$. Let $0 \rightarrow A \rightarrow I_0 \rightarrow I_1 \rightarrow ...$ be a minimal injective resolution of $A$, \newline then $B:=End_A(I_0 \oplus \Omega^{-n}(A))$ has dominant dimension $n$.
\end{theorem}
\begin{proof}
Let $e_iA$, $i=1,...,l$ be the indecomposable projective noninjective modules.
Note that $\Omega^{-n}(A) \cong \bigoplus_{k=1}^{l}{\Omega^{-n}(e_iA)}$. Since $domdim(A)=injdim(A)$, every noninjective indecomposable projective module must have dominant dimension equal to its injective dimension: $domdim(e_iA)=injdim(e_iA)=n$.
Now we look at a minimal injective resolution of such a noninjective module $e_iA$: \newline
$0 \rightarrow e_iA \rightarrow I_0 \rightarrow I_1 \rightarrow ... \rightarrow I_{n-1} \rightarrow \Omega^{-n}(e_iA) \rightarrow 0$. Note that all $I_j$ are injective and projective and $\Omega^{-n}(e_iA)$ is injective but not projective, since $e_iA$ has dominant dimension and injective dimension equal to $n$. The above minimal injective resolution also is a minimal projective resolution of $\Omega^{-n}(e_iA)$ ending with the projective and noninjective module $e_iA$. Thus $\Omega^{-n}$ induces an equivalence between the subcategory of projective modules having no nonzero injective direct summand to the subcategory of injective modules having no nonzero projective direct summand with inverse $\Omega^{n}$.
Then $End_A(eA \oplus \Omega^{-n}(A))$ is Morita equivalent to $End_A(D(A)) \cong A$, since $add(eA \oplus \Omega^{-n}(A))=add(D(A))$. Thus $domdim(End_A(eA \oplus \Omega^{-n}(A)))$ has the same dominant dimension as $A$.
\end{proof}
Since the algebras $A_n$ with $n \equiv 0$ mod $n$ are higher Auslander algebras by \hyperref[domdimA]{ \ref*{domdimA}}, they have property * using the previous theorem.
The class of algebras having property * contains interesting classes of algebras like higher Auslander algebras and Morita algebras and it might be an interesting question whether one can give another equivalent condition to describe those algebras.
In the following we denote by $Dom_2$ the subcategory of $mod-A$ consisting of modules having dominant dimension at least two. By $Codom_2$ we denote the subcategory of $mod-A$ consisting of modules having codominant dimension at least two.
\begin{lemma}
\label{APT}
Let $A$ be a Morita algebra with minimal faithful injective-projective module $eA$ and let $F:=(-)e: mod-A \rightarrow mod-eAe$ be the functor induced by right multiplication with $e$.
Assume $X$ has the property that $domdim(X)+codomdim(X) \geq 1$, then $F$ induces an isomorphim between $Hom_A(X,X)$ and $Hom_{eAe}(Xe,Xe)$.
\end{lemma}
\begin{proof}
This is a special case of Lemma 3.1. (2) in \cite{APT}. 
\end{proof}

\begin{proposition}
\label{HuXi}
Let $A$ be a selfinjective algebra and $M$ a nonprojective module.
Then for every $i \in \mathbb{Z}$ the algebras $End_A(A \oplus M)$ and $End_A(A \oplus \Omega^{i}(M))$ have the same dominant dimensions, finitistic dimensions and global dimensions.
Furthermore they also have the same Gorenstein dimension. 
\end{proposition}
\begin{proof}
This is a special case of corollary 1.2. and corollary 1.3. (2) in \cite{HuXi}.
The statement about the Gorenstein dimension can not be found in \cite{HuXi}, while it can be easily proven with the tools from this paper. Alternatively, it follows from the formula for the Gorenstein dimension in \cite{CheKoe}, Proposition 3.11. 
\end{proof}

In \cite{CX} it is proved that algebras of the form $B=End_A(eA \oplus \Omega^{-i}(A))$ have the same dominant dimension as $A$ in case $A$ is a Morita algebra and $i \in \{1,2,...,domdim(A) \}$.
Part 1 of the next proposition gives an alternative proof of this fact in case $A$ is a gendo-symmetric algebra.
\begin{proposition}
Let $A$ be a nonsymmetric, gendo-symmetric algebra with finite dominant dimension $n$ and minimal faithful projective-injective module $eA$. Let $i \in \{1,2,...,n\}$ and $B_i:=End_A(eA \oplus \Omega^{-i}(A))$. \newline
1.$B_i$ has the same dominant dimension as $A$. \newline
2.$B_i$ has the same finitistic dimension as $A$. \newline
3.$B_i$ has the same Gorenstein dimension as $A$. \newline
4.$B_i$ has the same global dimension as $A$.
\end{proposition}
\begin{proof}
Using $\Omega^{-i}(A) \cong \Omega^{n-i}(\Omega^{-n}(A))$ and the fact that gendo-symmetric algebras have dominant dimension at least two, we see that any module of the form $L=eA \oplus \Omega^{-i}(A)$ has the property that $domdim(L)+codomdim(L) \geq 2$.
When speaking about homological dimensions in the following, we always mean the dominant,global,Gorenstein or finitistic dimension.
Using \hyperref[APT]{ \ref*{APT}}, we see that $End_A(eA \oplus \Omega^{-i}(A)) \cong End_{eAe}(eAe \oplus \Omega^{-i}(A)e)$. Since the functor $F=(-)e$ is exact, one has $\Omega^{-i}(A)e \cong \Omega^{-i}(Ae)$ because $i \leq n$ and $(-)e$ sends projective-injective $A$-modules to injective $eAe$-modules. Using \hyperref[HuXi]{\ref*{HuXi}}, we see that $End_A(eA \oplus \Omega^{-i}(A)) \cong End_{eAe}(eAe \oplus \Omega^{-i}(A)e) \cong End_{eAe}(eAe \oplus \Omega^{-i}(Ae))$ and $End_{eAe}(eAe \oplus Ae)$ have the same homological dimensions. Now we use that $A$ is a gendo-symmetric algebra and thus $A \cong End_{eAe}(Ae)$. Note that $End_{eAe}(Ae)$ is Morita-equivalent to $End_{eAe}(eAe \oplus Ae)$, since $add(eAe) \subseteq add(eAe \oplus Ae)$ because $Ae \cong (1-e)Ae \oplus eAe$ (as $eAe$-modules) is a generator of $mod-eAe$. Thus since the relevant homological dimensions are invariant under Morita-equivalence, also $End_A(eA \oplus \Omega^{-i}(A))$ has the same homological dimensions as $A$. 
\end{proof}

The previous proposition can be used to get new higher Auslander-Solberg algebras from oldes ones using canonical tilting modules:
\begin{example}
Let $n \geq 1$ and $d \geq 1$.
Let $A$ be the Nakayama algebra with Kupisch series $[2d,2d+1]$ and minimal faithful projective-injective module $eA$.
It can be easily seen that $A \cong End_C(C \oplus C/soc(C))$, where $C:=K[x]/(x^{d+1})$ is a symmetric Nakayama algebra with one simple module.
It is easy to check that $A$ has dominant dimension and Gorenstein dimension equal to $2$. Furthermore $A$ has finite global dimension iff $d=1$. Thus in general $A$ is a higher Auslander-Solberg algebra and a higher Auslander algebra iff $d=1$. Now let $B:=End_A(eA \oplus \Omega^{-1}(A)).$ By the above proposition, this algebra is again a higher Auslander-Solberg algebra with the same dominant and Gorenstein dimension as $A$. For $d=1$, $B$ is isomorphic to $A$.
The quiver with relations $I$ of the algebra $B:=End_A(eA \oplus \Omega^{-1}(A))$ looks as follows for $d \geq 2$:
$$
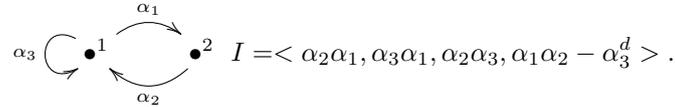
\begin{xy}
  \xymatrix{
      & {\bullet}^1\ar@(ul,dl)_{\alpha_3} \ar@/ ^1pc/[r]^{\alpha_1}   & {\bullet}^2 \ar@/ ^1pc/[l]^{\alpha_2}}
\end{xy} \newline 
\\\ I=<\alpha_2 \alpha_1, \alpha_3 \alpha_1, \alpha_2 \alpha_3, \alpha_1 \alpha_2 - \alpha_3^d>.
$$ 
\end{example}
The next example shows that the previous proposition is not true for general algebras with dominant dimension at least two.
\begin{example}
The following example shows that the previous proposition is wrong in case the algebra is not assumed to be a gendo-symmetric algebra:
Let $A$ be the Nakayama algebra with Kupisch series $[3,3,3,3,2,1]$. We saw in \hyperref[domdimA]{ \ref*{domdimA}}, that $A$ has dominant and global dimension equal to 3 and thus is a higher Auslander algebra. The algebra $B:=End_A(eA \oplus \Omega^{-1}(A))$ is isomorphic to the following algebra with quiver and relations: \newline
$
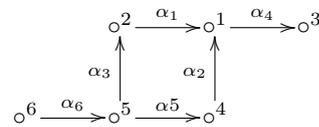
\xymatrix@1{  & \circ^2 \ar[r]^{\alpha_1} & \circ^1 \ar[r]^{\alpha_4} & \circ^3 & & &  &  \\ \circ^{6}\ar[r]^{\alpha_6} & \circ^{5}\ar[r]^{\alpha5} \ar[u]^{\alpha_3} & \circ^4 \ar[u]^{\alpha_2} & &  & & } \\\ I=<\alpha_3 \alpha_1-\alpha_5 \alpha_2, \alpha_6 \alpha_3>$. 
It can easily be checked that $B$ has dominant dimension 1 and global dimension 2 and thus is not a higher Auslander algebra.
\end{example}

\end{document}